\documentclass[letterpaper, 10 pt, conference]{ieeeconf}
\IEEEoverridecommandlockouts
\usepackage{cite}

\usepackage{amsmath,amssymb,amsfonts,bm}
\usepackage{enumerate}
\usepackage{graphicx}
\usepackage{textcomp}
\usepackage{xcolor}
\usepackage{color}
\usepackage{amsthm}
\usepackage{amsmath}
\usepackage{algorithm}
\usepackage{algpseudocode}

\usepackage{tikz}
\usepackage{dsfont}
\newtheorem{thm}{Theorem}

\newtheorem{cor}{Corollary}
\newtheorem{lem}{Lemma}

\theoremstyle{definition}
\newtheorem{ass}{Assumption}

\newtheorem{defn}{Definition}
\newtheorem{rem}{Remark}

\newtheorem{prob}{Problem}
\usepackage[utf8]{inputenc}
\usepackage{color}
\usepackage[titles]{tocloft}

\usepackage{hyperref}

\usepackage[caption=false, font=footnotesize]{subfig}
\begin{document}

\title{Linearly Solvable Continuous-Time General-Sum \\ Stochastic Differential Games}

\author{Monika Tomar \and Takashi Tanaka 
\thanks {M. Tomar and T. Tanaka are with the Edwardson School of Industrial Engineering, School of Aeronautics and Astronautics, Elmore Family School of Electrical and Computer Engineering, Purdue University, West Lafayette, IN, USA, respectively.
Email: \href{mailto:tomarm@purdue.edu}{tomarm@purdue.edu},
\href{mailto:tanaka16@purdue.edu}{tanaka16@purdue.edu}}
\thanks{This work is supported by AFOSR DSCT program under grant FA9550-25-1-0347 and DARPA COMPASS program under grant HR0011-25-3-0210.}}
\maketitle
\begin{abstract}
This paper introduces a class of continuous-time, finite-player stochastic general-sum differential games that admit solutions through an exact linear PDE system. We formulate a distribution planning game utilizing the cross-log-likelihood ratio to naturally model multi-agent spatial conflicts, such as congestion avoidance. By applying a generalized multivariate Cole-Hopf transformation, we decouple the associated nonlinear Hamilton-Jacobi-Bellman (HJB) equations into a system of linear partial differential equations. This reduction enables the efficient, grid-free computation of feedback Nash equilibrium strategies via the Feynman-Kac path integral method, effectively overcoming the curse of dimensionality.
\end{abstract}
\vspace{-2mm}
\section{Introduction}
\vspace{-1mm}
Stochastic games provide a natural framework for modeling interacting decision makers under uncertainty, and they arise in control, economics, traffic systems, and networked multi-agent settings. In such problems, the feedback Nash equilibrium is especially important because it is a closed-loop, state-dependent, and strongly time-consistent solution concept~\cite{BasarOlsder1998}. The difficulty, however, is that computing feedback Nash equilibria in stochastic differential games typically leads to coupled nonlinear Hamilton-Jacobi-Bellman (HJB) or Hamilton-Jacobi-Bellman-Isaacs (HJBI) equations. Consequently, even when existence and characterization results are available~\cite{buckdahn2008stochastic}, the resulting equilibrium PDE systems are often analytically intractable and numerically challenging. To avoid grid-based computation, \cite{cohen2017nash} maps infinite-horizon ergodic games to a system of coupled Ergodic BSDEs which still necessitate complex forward-backward solvers.  For a risk-sensitive ergodic setup, the solution of the coupled HJBs is characterized in  \cite{ghosh2023nonzero} via a multi-parameter eigenvalue problem.

One of the main exceptions to this general picture is the linearly solvable or Kullback-Leibler (KL) control framework. In a single agent control problem, \cite{todorov2006linearly} formulated a Markov Decision Process in the discrete-time setting, where control is modeled as a change of transition probabilities penalized by a KL divergence, the Bellman equation becomes linear after an exponential desirability transformation. A corresponding continuous-time version of the stochastic optimal control problem was developed in \cite{kappen2005linear} where the resultant nonlinear HJB was linearizable by a similar nonlinear transform (namely Cole-Hopf Transformation) which allows for the transformed linear problem to admit Monte Carlo simulations of Path Integrals. An explicit connection was established between the MDPs and the Path Integral problem in \cite {theodorou2012relative}. Subsequent work extended this line of research to various game-theoretic settings. For instance, a KL cost based Markov game in an adversarial zero-sum discrete-time setting is  modeled in \cite{dvijotham2012linearly}, where the resultant Bellman Equation was shown to be linearizable without needing a change of variables.  In multi-agent settings, linearly solvable structure has been established in special classes such as mean-field traffic routing games \cite{tanaka2020linearly} where the equilibrium for the mean-field game in the discrete-time setting was shown equivalent to a single Bellman equation which was linearized using a Cole-Hopf transformation. Another recent work \cite{soper2024linearly} expanded on the idea of using the log-transform to linearize a general-sum discrete-time game, where the costs are KL costs and the players control the probabilistic transitions of passive dynamics of a common underlying MDP. In a continuous-time setting, a two-player zero-sum stochastic differential game is modeled that is linearly solvable via the path-integral control method \cite{patil2023risk}. A multi-agent cooperative problem is developed in \cite{wan2021cooperative}, where agents optimize for a common objective within a local subsystem and hence a single PDE is solved using the path-integral method. To the best of our knowledge, there is no general-sum continuous-time stochastic differential game formulation that is linearly solvable via the Path Integral approach.

Motivated by the success of the path integral approach in computational advances of the HJB equation, this paper introduces a class of nonlinear stochastic general-sum differential games for a finite number of heterogeneous players that become linearly solvable via the Path integral approach. Through an equivalent information theoretic representation, the proposed setup models a measure-theoretic planning game in which players select controlled probability distributions over their trajectories. Each player's objective balances individual costs, KL divergence from baseline distributions, and cross-log-likelihood terms coupling the agent's measures. Broadly, these cross terms regulate interactions over shared resources, driving emergent behaviors that range from mutual resource partitioning to aggregation, and more generally to asymmetric interactions when pairwise couplings are not symmetric. When formulated to penalize distributional overlap, a practical application of this mechanism is congestion avoidance. There are many ways of modeling congestion avoidance explored in the literature, primarily encoding interactions through macroscopic density effects or route/lane occupancy \cite{pedram2019linearly,festa2018mean}, or microscopically through pairwise geometric separation and proximity penalties \cite{mylvaganam2017differential}, or barrier-function constraints \cite{pereira2022decentralized}. The cross-log-likelihood structure in our formulation targets the macroscopic setting by penalizing overlap of the agents' probability measures directly: an agent incurs a high cost for assigning probability mass to trajectories that other agents also heavily favor, while being biased towards available cost-effective reference distributions. Congestion is therefore resolved at the distributional planning stage, leading to emergent distributional separation and proactive congestion avoidance while preserving exact linearizability of the coupled HJB system.

The paper is organized as follows: We formulate the measure-theoretic general-sum game with cross-log-likelihood interactions, establish its equivalence to a nonlinear stochastic differential game, and derive the coupled HJB equations for feedback Nash equilibrium. We then introduce a multivariate Cole-Hopf transformation that decouples and linearizes the entire system, enabling solution via forward Monte Carlo sampling through the Feynman-Kac formula. Finally, we validate the framework on an asymmetric multi-player collision-avoidance scenario demonstrating emergent distributional separation among the agents. The next section formalizes the game and its equivalent stochastic differential-game representation.

\section{Problem Formulation}
\vspace{-1pt}
We consider an $N$-player continuous-time dynamic general-sum game over a finite time horizon $[0, T]$. Each player deploys a team of identical microscopic agents to a common state space $\mathcal{X} = \mathbb{R}^n$. Let $\Omega := \mathcal{C}([0,T], \mathbb{R}^n)$ denote the path space of continuous state trajectories. We assume that the dynamics of different microscopic agents are decoupled. Let the state for each agent $x_t \in \mathcal{X}$ follow the It\^o stochastic differential equation (SDE):
\vspace{-1mm}
\begin{align}
    d x_t = f(x_t)\,dt + g(x_t)\,d v_t, \qquad 0 \le t \le T \label{eq:base_dyanmics}
\end{align}
where $v_t \in \mathbb{R}^m$ is an exogenous input containing both the control drift and stochastic disturbances. We assume that $f: \mathbb{R}^n \rightarrow \mathbb{R}^n$ and $g: \mathbb{R}^n \rightarrow \mathbb{R}^{n \times m}$ are sufficiently regular to ensure the existence of a unique strong solution of \eqref{eq:base_dyanmics}, see \cite[Thm 5.2.1]{oksendal2003stochastic}. All agents in the same team are driven by the same control law, but individually they are subject to independent stochastic disturbances. For agents belonging to player $i$'s team, where $i \in \mathcal{N} := \{1, \dots, N\}$, we first describe the controlled dynamics under player $i$'s chosen probability measure $P^i$ on $\Omega$. Under $P^i$, the input process evolves as
\vspace{-2mm}
\begin{align}
    d v_t = u_t^i\,dt + d w_t^i,
    \label{eq:controlled_dyamics}
\end{align}
where $u_t^i$ is the feedback control selected by player $i$, and $w_t^i$ is an $m$-dimensional standard Wiener process under $P^i$.

Next, let $\bar u_t^i$ denote a nominal feedback policy for player $i$, and define the process $\bar w_t^i$ by
\begin{align}
    d\bar w_t^i := (u_t^i - \bar u_t^i)\,dt + d w_t^i
    \label{eq:shifted_noise}
\end{align}
Under a Novikov condition \cite[Thm 8.6.4, eq 8.6.8]{oksendal2003stochastic}, there exists a reference (or baseline) probability measure $R^i$ on $\Omega$ under which $\bar w_t^i$ is a $m$-dimensional standard Wiener process such that $P^i$ is absolutely continuous with respect to $R^i$ (denoted by $P^i \ll R^i$). Equivalently, under the reference measure $R^i$, the input admits the representation
\vspace{-1mm}
\begin{align}
    d v_t = \bar u_t^i\,dt + d\bar w_t^i
    \label{eq:nominal_input}
    \vspace{-1mm}
\end{align}
In this sense, player $i$'s strategy may be understood either as the choice of a controlled measure $P^i$ or, equivalently, as the choice of a feedback control $u_t^i$ relative to the nominal pair $(R^i,\bar u_t^i)$. Let $\boldsymbol{P} = (P^1, \dots, P^N)$ denote the joint strategy profile. For a given profile $\boldsymbol{P}$, Player $i$ incurs the cost:
\begin{align}
    \mathcal{J}^i(\boldsymbol{P}) :=  \mathbb{E}_{P^i}\!\left[ \int_0^T C_t^i(x_t)\,dt + \Psi^i(x_T) \right] +  \nonumber \\ \alpha_{ii}\,\mathbb{E}_{P^i}\!\left[ \log\frac{dP^i}{dR^i}(\omega) \right] + \sum_{j \neq i} \alpha_{ij}\,\mathbb{E}_{P^i}\!\left[ \log\frac{dP^j}{dR^j}(\omega) \right] \label{eq:distribution_game}
\end{align}
where $dP^i/dR^i$ is the Radon-Nikodym
derivative, $\omega \in \Omega$ represents a sample trajectory, $x_t(\omega)$ is the state at time $t$, $C_t^i : \mathcal{X} \to \mathbb{R}$ is the running cost, and $\Psi^i : \mathcal{X} \to \mathbb{R}$ is the terminal cost. Throughout, superscripts index players and subscripts denote time except for matrix entries $\alpha_{ij}$ and $\beta_{ij}$. The argument $(t,x)$ is omitted when clear from context. Consequently, given $P^{-i}:=(P^1,\dots,P^{i-1},P^{i+1},\dots,P^N)$, the player $i$ solves
\vspace{-1mm}
\[
\inf_{P^i \ll R^i} \mathcal{J}^i(P^i,P^{-i})
\]
\vspace{-1mm}
The weighting parameters $\alpha_{ij}$ represent the interactions between the players. We define the interaction matrix $\alpha \in \mathbb{R}^{N\times N}$ and its inverse as
\setlength{\belowdisplayskip}{3pt}

\begin{align}
    \alpha =
\begin{bmatrix}
    \alpha_{11} & \cdots & \alpha_{1N}\\
    \vdots      & \ddots & \vdots\\
    \alpha_{N1} & \cdots & \alpha_{NN}
\end{bmatrix},
\qquad \alpha_{ii} > 0,
\qquad \beta := \alpha^{-1} \label{eq:alpha_matrix}
\end{align}
where we assume $\alpha$ is non-singular. The objective functional encapsulates three primary behaviors. The first term  represents the standard expected trajectory cost. The second term is a self-KL divergence that penalizes player $i$'s deviation from its nominal plan, effectively acting as a control effort penalty.
The third term couples players through the likelihood ratios, $dP^j/dR^j$, which measures how heavily player $j$ targets a trajectory relative to its baseline $R^j$. For repulsive interactions ($\alpha_{ij} > 0$), player $i$ minimizes cost by avoiding trajectories where this ratio is high (strategies heavily favored by $j$) and shifting toward trajectories where it is low (cost-effective strategies of $R^j$ vacated by $j$). Intuitively, this structure drives proactive conflict avoidance (aggregation if $\alpha_{ij} <0$) while biasing these evasion strategies towards efficient nominal plans. Since the interaction matrix is not restricted to be symmetric, these incentives also capture asymmetric objectives. A similar coupling cost was considered in \cite{pedram2019linearly} as a tax to induce congestion avoidance behavior.

\begin{prob}
Find a Nash Equilibrium strategy profile $\boldsymbol{P}^* = (P^{1*}, \dots, P^{N*})$ for game \eqref{eq:distribution_game} such that, for all players $i \in \mathcal{N}$ and any admissible measure $P^i \ll R^i$, the following inequality holds:
\begin{equation}
    \mathcal{J}^i(P^{i*}, P^{-i*}) \leq \mathcal{J}^i(P^i, P^{-i*}),
    \label{eq:nash_inequality}
\end{equation}
given the fixed reference measures $R^i$, the interaction matrix $\alpha$ in \eqref{eq:alpha_matrix}, and where strategies $P^i$ are induced by closed-loop state feedback policies.
\end{prob}
\vspace{-3.5mm}
\section{Main Results}
\label{sec:main_results}
\vspace{-3pt}
We now state the first result, which translates the abstract measure-theoretic game \eqref{eq:distribution_game} into an equivalent nonlinear stochastic differential game with explicit control costs.

\begin{thm}
\label{thm:problem_equivalence}
Subject to the dynamics \eqref{eq:base_dyanmics}--\eqref{eq:controlled_dyamics}, the measure-theoretic game given by \eqref{eq:distribution_game} is equivalent to the following stochastic differential game, where each player $i \in \mathcal{N}$ minimizes:
\setlength{\abovedisplayskip}{3pt}
\begin{align}
    \min_{u^i}\; &\mathbb{E}_{P^i}\! \Bigg[ \int_0^T \Bigg( C_t^i(x_t) + \frac{\alpha_{ii}}{2} \|u_t^i - \bar{u}_t^i\|^2 \nonumber \\
    &+ \sum_{j \neq i} \Big\{ \frac{\alpha_{ij}}{2}(\|\bar{u}_t^j\|^2 - \|u_t^j\|^2) + \alpha_{ij} (u_t^j - \bar{u}_t^j)^\top u_t^i \Big\}  \Bigg) dt \nonumber \\
    &+ \Psi^i(x_T) \Bigg]
    \label{eq:equivalent_cost}
\end{align}
\end{thm}

\begin{proof}
    The proof is structured in two parts: determining the explicit control cost equivalent to the self-KL divergence term, and then evaluating the cross divergence term under the measure $P^i$.
    
    \textit{Step 1: Control Cost due to Self-KL Divergence.} 
    Using Girsanov's theorem \cite{oksendal2003stochastic} and \eqref{eq:shifted_noise}, the Radon-Nikodym derivative of $P^i$ with respect to $R^i$ is given by:
    \begin{equation}
        \log\frac{dP^i}{dR^i}(\omega) = \int_0^T (u_t^i - \bar{u}_t^i)^\top d w_t^i + \frac{1}{2}\int_0^T \|u_t^i - \bar{u}_t^i\|^2\,dt
        \label{eq:rn_self}
    \end{equation}
    Taking the expectation under $P^i$, the stochastic integral with respect to $w_t^i$ vanishes due to the martingale property, yielding:
    \begin{equation}
        \mathbb{E}_{P^i}\!\left[\log\frac{dP^i}{dR^i}(\omega)\right] = \frac{1}{2}\,\mathbb{E}_{P^i}\!\left[\int_0^T \|u_t^i - \bar{u}_t^i\|^2\,dt\right]
        \label{eq:self_kl_result}
    \end{equation}
    This indicates that the relative entropy cost is equivalent to the mean-square deviation from the nominal input $\bar u_t^i$.

    \textit{Step 2: Coupling Cost due to Cross-Log-Likelihood.}
    For any player $j \neq i$, the log Radon-Nikodym derivative between their controlled measure $P^j$ and the reference measure $R^j$ is derived analogously to \eqref{eq:rn_self}, utilizing the standard Wiener process $d\bar{w}_t^j$ under $R^j$:
    \begin{equation}
        \log\frac{dP^j}{dR^j}(\omega) = \int_0^T (u_t^j - \bar{u}_t^j)^\top d\bar{w}_t^j - \frac{1}{2}\int_0^T \|u_t^j - \bar{u}_t^j\|^2\,dt
        \label{eq:rn_cross_initial}
    \end{equation}
    Under the expectation of player $i$'s measure $P^i$, the system evolves according to $d v_t = u_t^i\,dt + d w_t^i$. We can rewrite $d\bar{w}_t^j$ in terms of player $i$'s processes:
    
    \begin{equation}
        d\bar{w}_t^j = d v_t - \bar{u}_t^j\,dt = (u_t^i - \bar{u}_t^j)\,dt + d w_t^i
        \label{eq:wiener_substitution}
    \end{equation}
    Substituting \eqref{eq:wiener_substitution} into \eqref{eq:rn_cross_initial} gives:
    \begin{align}
        \log\frac{dP^j}{dR^j}(\omega) &= \int_0^T (u_t^j - \bar{u}_t^j)^\top \big((u_t^i - \bar{u}_t^j)\,dt + d w_t^i\big) \nonumber \\
        &\quad - \frac{1}{2}\int_0^T \|u_t^j - \bar{u}_t^j\|^2\,dt
        \label{eq:rn_cross_subbed}
    \end{align}
    Taking the expectation under $P^i$, the stochastic integral $\int (\cdot)^\top dw_t^i$ vanishes. Consolidating the remaining terms inside the deterministic integral yields:
\vspace{-5mm}

    \begin{align}
        &\mathbb{E}_{P^i}\!\left[\log\frac{dP^j}{dR^j}(\omega)\right] \nonumber \\
        &= \mathbb{E}_{P^i}\!\left[ \int_0^T \left( (u_t^j - \bar{u}_t^j)^\top (u_t^i - \bar{u}_t^j) - \frac{1}{2}\|u_t^j - \bar{u}_t^j\|^2 \right) dt \right] \nonumber \\
        &= \mathbb{E}_{P^i}\!\left[ \int_0^T \left( \frac{1}{2}\left( \|\bar{u}_t^j\|^2 - \|u_t^j\|^2 \right) + (u_t^j -\bar{u}_t^j)^\top u_t^i\, \right)dt \right]
        \label{eq:cross_kl_result}
    \end{align}
    \vspace{-1mm}
    Substituting \eqref{eq:self_kl_result} and \eqref{eq:cross_kl_result} back into the original objective \eqref{eq:distribution_game} recovers the equivalent cost functional \eqref{eq:equivalent_cost}.
\end{proof}

\begin{rem}
\label{rem:control_equivalence}
Theorem \ref{thm:problem_equivalence} gives an equivalent parametrization of each player’s strategy: relative to a fixed nominal pair ($R^i, \bar u^i$), a feedback control $u
^i$ induces an absolutely continuous path measure $P^i$, and conversely, any admissible measure $P^i \ll R^i$ determines the associated unique drift adjustment. The coupling term $\alpha_{ij}\,\mathbb{E}_{P^i}[\log(dP^j/dR^j)]$ therefore expresses how player $i$ evaluates player $j$’s likelihood ratio on trajectories sampled from $P^i$; in the control representation, this becomes the explicit cross term in $u^i$ and $u^j$, so the interaction is at the level of relative measure changes on the common path space.
\end{rem}
Before analyzing the solutions, we characterize the Feedback Nash Equilibrium via the coupled Hamilton-Jacobi-Bellman (HJB) equations. 

\begin{lem}
\label{lem:coupled_hjb}
Consider the stochastic differential game defined in Theorem \ref{thm:problem_equivalence} subject to the dynamics \eqref{eq:base_dyanmics}--\eqref{eq:controlled_dyamics}. Let $J_t^i(x)$ be Player $i$'s value function on $[0,T] \times \mathcal{X}$. Then, the Feedback Nash Equilibrium is governed by the following system of coupled nonlinear HJB equations for $i \in \{1, \dots, N\}$:
\begin{align}
    &-\partial_t J^i = C_t^i(x) + (f(x)+g(x)\bar u^i)^\top \nabla_x J^i + \nonumber \\
    &\frac{1}{2} \mathrm{Tr}\left(g(x) g(x)^\top \nabla_{xx}^2 J^i\right) + \nonumber \\
    &\sum_{j \neq i}\alpha_{ij}\left(\bar u^j-\bar u^i\right)^\top g(x)^\top \sum_{k=1}^N\beta_{jk}\nabla_xJ^k \nonumber \\
     &- \frac{1}{2} \nabla \boldsymbol{J}^\top (\beta^\top \otimes g(x))\, \mathrm{diag}_{1\le j\le N}(\alpha_{ij} I_m)\, (\beta \otimes g(x)^\top) \nabla \boldsymbol{J}, \label{eq:hjb_pde} \\
    &J_T^i(x) = \Psi^i(x), \label{eq:hjb_boundary}
\end{align}
where $\nabla \boldsymbol{J} := [\nabla_x {J^1}^\top, \dots, \nabla_x {J^N}^\top]^\top \in \mathbb{R}^{Nn}$ is the stacked gradient vector.
\end{lem}
\vspace{-2mm}
\begin{proof}
    Applying the dynamic programming principle to the equivalent cost in Theorem \ref{thm:problem_equivalence}, the HJB equation for Player $i$ is given by:
    
    \begin{align}
        -\partial_t J^i &= \min_{u^i} \Bigg[ C_t^i(x) + \sum_{j=1}^N \frac{\alpha_{ij}}{2} (u^j-\bar u^j)^\top (2u^i - u^j -\bar u^j) \nonumber \\
        &\qquad\quad + (f(x) + g(x)u^i)^\top \nabla_x J^i \nonumber \\
        &\qquad\quad + \frac{1}{2} \mathrm{Tr}\left(g(x) g(x)^\top \nabla_{xx}^2 J^i\right) \Bigg],
        \label{eq:hjb_pre_min}
    \end{align}
    
    with the terminal condition $J_T^i(x) = \Psi^i(x)$.

    Since $\alpha_{ii}>0$, the function on the right-hand side of \eqref{eq:hjb_pre_min} is convex in $u^i$. Taking the derivative of the RHS with respect to $u^i$ yields the first-order necessary and sufficient optimality condition:
    \begin{equation}
        \sum_{j=1}^N \alpha_{ij} (u^j -\bar u^j) + g(x)^\top \nabla_x J^i = 0, \qquad i\in\mathcal{N}.
        \label{eq:hjb_foc}
    \end{equation}
    Stacking the conditions from \eqref{eq:hjb_foc} for all $N$ players into a single linear system allows us to solve for the optimal feedback policies collectively. Defining the stacked optimal control $\boldsymbol{u}^* = [(u^{1*} - \bar u^1)^\top, \dots, (u^{N*} -\bar u^N)^\top]^\top$, we have $(\alpha \otimes I_m) \boldsymbol{u}^* = - (I_N \otimes g^\top) \nabla \boldsymbol{J}$. Multiplying by $\beta \otimes I_m$ yields the explicit equilibrium strategies:
    \begin{equation}
    \boldsymbol{u}^* = -(\beta \otimes g^\top)\nabla\boldsymbol{J}.
    \label{eq:optimal_u_stacked}
\end{equation}
    Plugging $\boldsymbol{u}^*$ from \eqref{eq:optimal_u_stacked} back into \eqref{eq:hjb_pre_min} directly yields the explicit coupled nonlinear PDEs in \eqref{eq:hjb_pde}.
\end{proof}

We now state the second main result, relating the solution of the coupled nonlinear HJB PDEs in Lemma \ref{lem:coupled_hjb} to a system of decoupled linear equations. To achieve this, we first introduce a change of variables.
\vspace{-1mm}
\begin{defn}[Multivariate Cole-Hopf Transformation]
\label{def:cole_hopf}
Let $Z_t^i(x)$ be the transformed desirability function for player $i$. We define the transformation mapping the value functions $J_t^i(x)$ to $Z_t^i(x)$ as:
\begin{equation}
\vspace{-0.5mm}
    \begin{bmatrix}
        J_t^1(x)\\
        \vdots\\
        J_t^N(x)
    \end{bmatrix}
    = -\alpha
    \begin{bmatrix}
        \log Z_t^1(x)\\
        \vdots\\
        \log Z_t^N(x)
    \end{bmatrix}
    \label{eq:cole_hopf}
\end{equation}
Equivalently, multiplying by $\beta = \alpha^{-1}$, we have $Z_t^i(x) = \exp\!\left(-\sum_{j=1}^N \beta_{ij} J_t^j(x)\right)$.
\end{defn}
\vspace{-2.5mm}
\begin{thm}
\label{thm:linear_pde_equivalence}
Under the transformation \eqref{eq:cole_hopf}, the system of coupled nonlinear HJB PDEs \eqref{eq:hjb_pde} with terminal conditions \eqref{eq:hjb_boundary} is equivalent to the following system of decoupled linear PDEs for each player $i \in \mathcal{N}$:
\begin{align}
    -\partial_t Z_t^i &= (f(x)+g(x)\bar u^i)^\top \nabla_x Z_t^i + \frac{1}{2} \mathrm{Tr}\big(g(x) g(x)^\top \nabla_{xx}^2 Z_t^i\big) \nonumber \\
    &\quad - \left(\sum_{j=1}^N \beta_{ij} C_t^j(x)\right) Z_t^i, \label{eq:linear_pde} \\
    Z_T^i(x) &= \exp\!\left(-\sum_{j=1}^N \beta_{ij}\Psi^j(x)\right)
    \label{eq:linear_boundary}
\end{align}
\end{thm}

\begin{proof}
    Differentiating the transformation \eqref{eq:cole_hopf} yields the following relations:
    \begin{align}
        \partial_t J^i &= -\sum_{j=1}^N \alpha_{ij} \frac{\partial_t Z^j}{Z^j}, \label{eq:dt_J} \\
        \nabla_x J^i &= -\sum_{j=1}^N \alpha_{ij} \frac{\nabla_x Z^j}{Z^j}, \label{eq:dx_J} \\
        \nabla_{xx}^2 J^i &= \sum_{j=1}^N \alpha_{ij} \left( \frac{\nabla_x Z^j (\nabla_x Z^j)^\top}{(Z^j)^2} - \frac{\nabla_{xx}^2 Z^j}{Z^j} \right) \label{eq:dxx_J}
    \end{align}
    
    Substituting \eqref{eq:dt_J}-\eqref{eq:dxx_J} into the HJB equation \eqref{eq:hjb_pde}, we observe that the term $\frac{1}{2}\mathrm{Tr}(gg^\top \nabla_{xx}^2 J^i)$ generates both linear terms (involving $\nabla_{xx}^2 Z^j$) and quadratic terms (involving $\nabla_x Z^j (\nabla_x Z^j)^\top$). 
    
    By the definition of the interaction matrix $\alpha$ and the equilibrium control structure, the explicit nonlinear cross-coupling term in the HJB \eqref{eq:hjb_pde} exactly cancels the quadratic terms generated by the Hessian of the logarithm. After this cancellation, the system reduces to:
    \vspace{-1mm}
    \begin{equation}
    \resizebox{\columnwidth}{!}{%
        $\alpha
        \begin{bmatrix}
            \frac{\partial_t Z^1}{Z^1} \\ \vdots \\ \frac{\partial_t Z^N}{Z^N}
        \end{bmatrix}
        = 
        \begin{bmatrix}
            C_t^1 \\ \vdots \\ C_t^N
        \end{bmatrix}
        -\alpha\begin{bmatrix}\frac{\left(f+g\bar u^1\right)^\top\nabla_x Z^1}{Z^1}\\ \vdots\\ \frac{\left(f+g\bar u^N\right)^\top\nabla_x Z^N}{Z^N}\end{bmatrix}
        - \frac{1}{2}\alpha
        \begin{bmatrix}
            \frac{\mathrm{Tr}(g g^\top \nabla_{xx}^2 Z^1)}{Z^1} \\ \vdots \\ \frac{\mathrm{Tr}(g g^\top \nabla_{xx}^2 Z^N)}{Z^N}
        \end{bmatrix}$
        }\label{eq:pre_linear_matrix}
    \end{equation}

    Left-multiplying \eqref{eq:pre_linear_matrix} by the inverse matrix $\beta = \alpha^{-1}$ decouples the system. Multiplying the $i$-th row by $Z_t^i$ yields the linear PDE \eqref{eq:linear_pde}. The terminal condition \eqref{eq:linear_boundary} directly follows from applying Definition \ref{def:cole_hopf} to the terminal cost $J_T^i(x) = \Psi^i(x)$.
\end{proof}

\begin{ass}
\label{ass:fk}
For each $i \in \mathcal{N}$, the functions $\sum_{j=1}^N \beta_{ij} C_t^j(x)$ and
$\sum_{j=1}^N \beta_{ij} \Psi^j(x)$ are continuous and bounded below.
\end{ass}
Since $\beta$ may have negative entries (e.g., under repulsive coupling),
Assumption~\ref{ass:fk} does not follow from $C^j$, $\Psi^j$ being bounded below.
Under it, $e^{-S_t^i}$ below is bounded, so $0 < Z_t^i(x) < \infty$ and
Definition~\ref{def:cole_hopf} is well posed.

\begin{cor}[Feynman-Kac Path Integral Solution]
\label{cor:feynman_kac}
Under Assumption ~\ref{ass:fk}, by the Feynman-Kac lemma, the solution to the linear PDE \eqref{eq:linear_pde} subject to the boundary condition \eqref{eq:linear_boundary} admits the following probabilistic path integral representation:
\vspace{-2mm}
\begin{align}
Z_t^i(x) = \mathbb{E}_{R^i}\!\big[ e^{-S_t^i(\omega)} \,\big|\, x_t = x \big]\label{eq:feynman_kac}
\end{align}

where
\begin{equation}
\vspace{-1mm}
    S_t^i(\omega) := \sum_{j=1}^N \beta_{ij}\Big[\int_t^T C_s^j(x_s)\,ds + \Psi^j(x_T)\Big]
    \label{eq:path_cost}
\end{equation}
is the interaction-adjusted path cost, and the expectation is taken with respect to the base measure $R^i$ corresponding to the nominally controlled forward dynamics
\vspace{-3mm}
\begin{equation}
    d x_s = f(x_s)\,ds + g(x_s)(\bar u_s^i\,ds +\,d \bar w_s^i)\label{eq:nominal_dynamics}
\end{equation}
\end{cor}
\vspace{-3mm}
\begin{rem}[Sampling Independence]
\label{rem:sampling}
The decoupled nature of the linear PDEs allows the expectation in \eqref{eq:feynman_kac} for each player to be evaluated independently of other players via forward Monte Carlo trajectory sampling from their own baseline measures.  This entirely bypasses the need for spatial discretization grids, thereby overcoming the curse of dimensionality typically associated with solving multi-agent HJB equations.
\end{rem}

\section{Optimal Control Computation and Measure Recovery}
\label{sec:optimal_control_measure}
\vspace{-2pt}
 In this section, we demonstrate how the optimal control can be evaluated via forward trajectory sampling, and we formally connect this result back to the original measure-theoretic general-sum game defined in \eqref{eq:distribution_game}.

\subsection{Path Integral Control via Monte Carlo Sampling}

Recall from the stacked first-order conditions \eqref{eq:optimal_u_stacked} and the Cole-Hopf transformation \eqref{eq:cole_hopf} that the optimal control for player $i$ is given by $u_t^{i*} = \bar u_t^i  + g(x_t)^\top \nabla_x \log Z_t^i(x_t)$, or equivalently, $u_t^{i*} = \bar u_t^i + \frac{1}{Z_t^i(x_t)} g(x_t)^\top \nabla_x Z_t^i(x_t)$. The following theorem establishes how this gradient can be computed directly from forward sampling without taking spatial derivatives.

\vspace{-1mm}
\begin{thm}[Path Integral Control under the Reference Measure]
\label{thm:path_integral_control}
Let $R_{t,x}^i$ and $S_t^i$ be as in Corollary~\ref{cor:feynman_kac}. Then the optimal feedback control satisfies
\vspace{-1mm}
\begin{equation}
    u_t^{i*}(x) = \bar{u}_t^i + \lim_{\delta t \to 0} \frac{1}{\delta t} \frac{\mathbb{E}_{R_{t,x}^i} \!\left[ e^{-S_t^i(\omega)} \delta \bar{w}_t^i \right]}{\mathbb{E}_{R_{t,x}^i} \!\left[ e^{-S_t^i(\omega)} \right]},
    \label{eq:pi_control_corrected}
\end{equation}
\vspace{-2pt}
where $\delta \bar{w}_t^i := \bar{w}_{t+\delta t}^i - \bar{w}_t^i$.
\end{thm}

\begin{proof}
By Corollary~\ref{cor:feynman_kac},  $ Z_t^i(x)
    =
    \mathbb E_{R_{t,x}^i}\!\left[e^{-S_t^i(\omega)}\right]$.
Moreover, the standard path-integral first-variation identity gives
\begin{equation}
    g(x)^\top \nabla_x Z_t^i(x)
    =
    \lim_{\delta t\to 0}
    \frac{1}{\delta t}
    \mathbb E_{R_{t,x}^i}\!\left[
        e^{-S_t^i(\omega)}\,\delta \bar w_t^i
    \right]
    \label{eq:gradient_noise_identity}
\end{equation}
Dividing by $Z_t^i(x)$ yields $g(x)^\top \nabla_x \log Z_t^i(x)
    =
    \lim_{\delta t\to 0}
    \frac{1}{\delta t}
    \frac{
        \mathbb E_{R_{t,x}^i}\!\left[
            e^{-S_t^i(\omega)}\,\delta \bar w_t^i
        \right]
    }{
        \mathbb E_{R_{t,x}^i}\!\left[
            e^{-S_t^i(\omega)}
        \right]
    }$. Substituting this into $ u_t^{i*}(x)=\bar u_t^i+g(x)^\top \nabla_x \log Z_t^i(x) $
proves \eqref{eq:pi_control_corrected}.
\end{proof}
\vspace{-1mm}
\begin{rem}
Theorem \ref{thm:path_integral_control} shows that the optimal control correction
\(u_t^{i*}(x)-\bar u_t^i\) is a weighted average of the reference noise realizations. Trajectories with lower interaction-adjusted cost $S_t^i(\omega)$ receive higher exponential weight and therefore exert greater influence on the optimal control update.
\end{rem}

\subsection{Recovery of the Optimal Probability Measure}

We now return to the original KL formulation. The objective is to identify the optimal path measure $P_{t,x}^{i*}\ll R_{t,x}^i$ induced by the optimal feedback control.
\vspace{-1mm}
\begin{thm}[Optimal Measure]
\label{thm:optimal_measure}
For each player $i$, let $P_{t,x}^{i*}$ denote the path measure induced by the optimal closed-loop control $u^{i*}$ starting from $x_t=x$. Then the optimal measure is the exponentially tilted reference measure
\begin{equation}
    \frac{dP_{t,x}^{i*}}{dR_{t,x}^i}(\omega)
    =
    \frac{e^{-S_t^i(\omega)}}{Z_t^i(x)},
    \qquad \forall \omega\in\Omega.
    \label{eq:optimal_measure_explicit_corrected}
\end{equation}
\end{thm}

\begin{proof}
Define the normalized process $M_s^i := \exp\!\big( \!-\!\int_t^s \sum_{j=1}^N \beta_{ij} C_r^j(x_r) \, dr \big) \frac{Z_s^i(x_s)}{Z_t^i(x)}$ for $s \in [t,T]$. 
Applying It\^o's lemma under the reference measure $R_{t,x}^i$ and substituting the linear PDE for $Z^i$, the drift terms cancel, yielding
\begin{align}
    dM_s^i &= M_s^i \, \big( g(x_s)^\top \nabla_x \log Z_s^i(x_s) \big)^\top d\bar w_s^i  \nonumber \\
    &= M_s^i \, \big( u_s^{i*}(x_s)- \bar u_s^i \big)^\top d\bar w_s^i,
\end{align}
where the second equality follows from the optimal control relation. Thus, $M_s^i$ is the density process (stochastic exponential) mapping the reference noise to the optimally controlled noise, meaning Girsanov's theorem gives $M_T^i = \frac{dP_{t,x}^{i*}}{dR_{t,x}^i}(\omega)$. 
Evaluating $M_T^i$ using the terminal condition $Z_T^i(x_T) = \exp\!\big( \!-\!\sum_{j=1}^N \beta_{ij}\Psi^j(x_T) \big)$ directly yields $M_T^i = e^{-S_t^i(\omega)} / Z_t^i(x)$, which completes the proof.
\end{proof}

\section{Simulation}

We illustrate the proposed framework using a two-player, one-dimensional game over the horizon $t \in [0,T]$ with state space $\mathcal{X}=\mathbb{R}$. This example is  designed to highlight the effect of the cross-log-likelihood coupling term in the game.

\subsection{Game Setup}

We consider a two-player scalar system $dx_t = \sigma\bigl(u_t^i(x_t)dt + dw_t^i\bigr)$ with moving quadratic costs $$C_t^i(x) = \frac{q}{2}\bigl(x-m^i_t\bigr)^2, \qquad \Psi^i(x) = \frac{q_T}{2}\bigl(x-m^i_T\bigr)^2,$$ where $m^1_t = -a(t/T)$ and $m^2_t = a(t/T)$. The players share a common initial state $x_0 = 0$ and an identical baseline reference process ($R$) $dx_t = \sigma dw_t$.
The interaction between the players is governed by the coupling matrix $\alpha(\gamma) = \left[\begin{smallmatrix} 1 & \gamma \\ \gamma & 1 \end{smallmatrix}\right]$ and its inverse $\beta(\gamma)=\alpha(\gamma)^{-1}$, where $|\gamma|<1$.
We evaluate the equilibrium using two methods. First, to recover the optimal measure $P^{i\ast}$ from the initial condition (Theorem \ref{thm:optimal_measure}), we draw an ensemble of trajectories under the common reference process $R$ and assign each trajectory $\omega$ the weight $\rho_i \propto \exp(-S^i(\omega))$, with $S^i$ given by \eqref{eq:path_cost}. Second, to compute the optimally controlled trajectories, we use the state-feedback law $u_t^{i\ast}(x)$  as given in Theorem \ref{thm:path_integral_control}. 
\begin{figure}[htbp]
    \centering
    \subfloat[Equilibrium measures.\label{fig:main_result}]{\includegraphics[width=\linewidth]{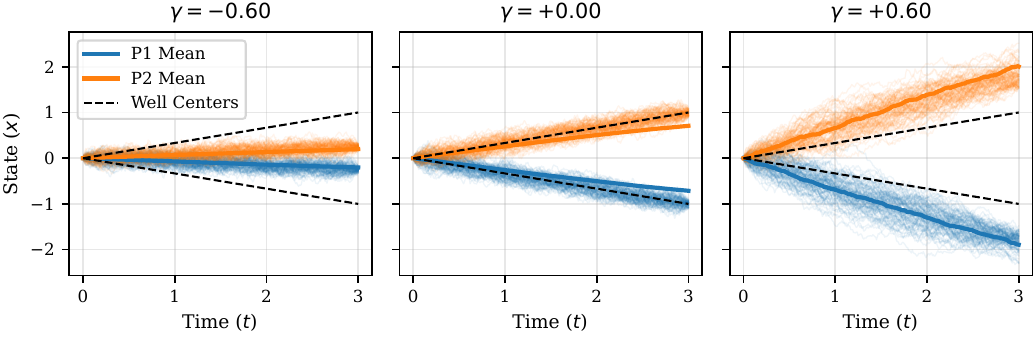}}
    \vspace{-1mm}
    \subfloat[Feedback controlled equilibrium trajectories.\label{fig:controlled_trajectories}]{\includegraphics[width=\linewidth]{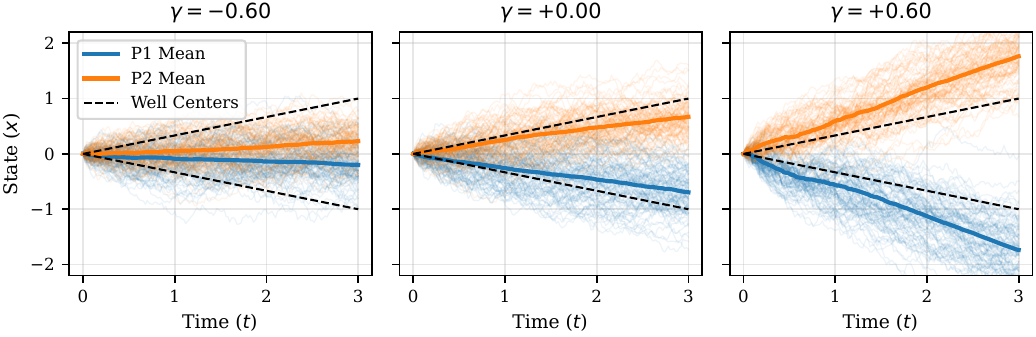}}
    \vspace{-1mm}
    \subfloat[Terminal densities.\label{fig:terminal_result}]{\includegraphics[width=\linewidth]{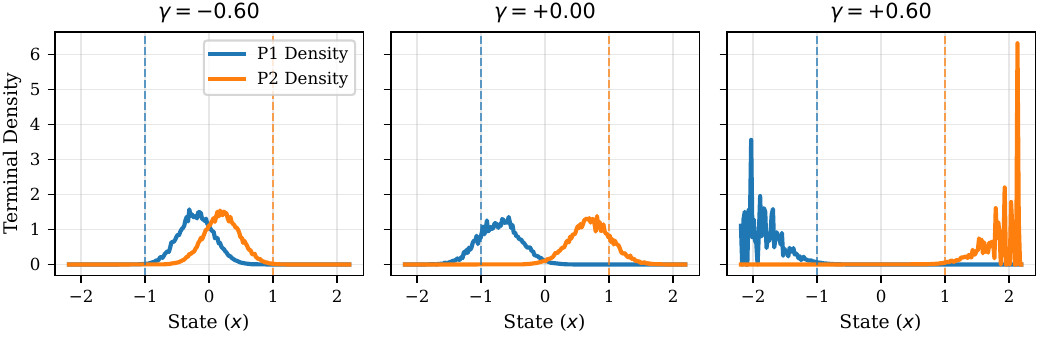}}
    \vspace{-1mm}
    \subfloat[Expectation distance.\label{fig:distance}]{\includegraphics[width=0.5\linewidth]{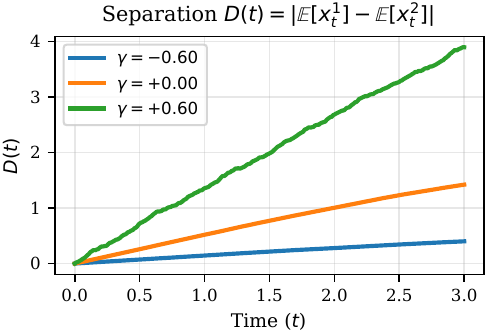}}
    \caption{Equilibrium behavior under symmetric coupling.}
    \label{fig:symmetric_results}
\end{figure}
\begin{figure}[htbp]
\centering
\includegraphics[width=\columnwidth]{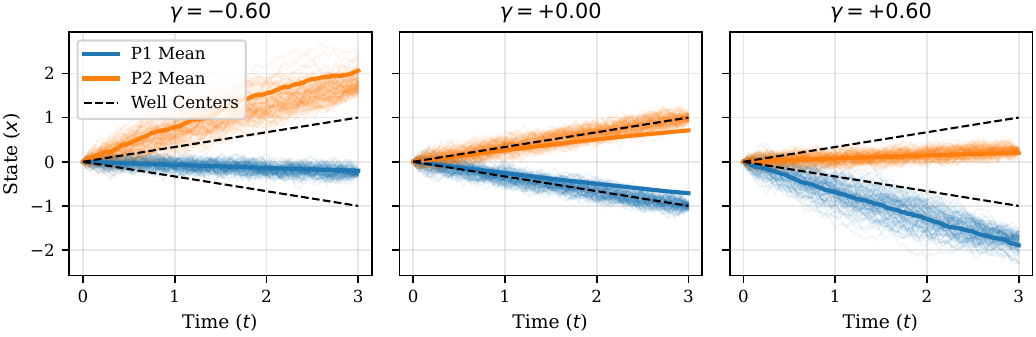}
\caption{ Equilibrium measures under asymmetric coupling.}
\label{fig:asymmetric_measures}
\vspace{-15pt}
\end{figure}

\subsection{Results}Figures \ref{fig:main_result} and \ref{fig:controlled_trajectories} demonstrate the distributional behavior across three distinct interaction regimes ($\gamma \in \{-0.6, 0.0, 0.6\}$). Notably, both figures represent the exact same Nash equilibrium computed via two different computational perspectives. Fig. \ref{fig:main_result} depicts the equilibrium measure obtained by reweighting uncontrolled reference trajectories, while Fig. \ref{fig:controlled_trajectories} illustrates the trajectories resulting from the closed-loop optimal control law. When uncoupled ($\gamma = 0.0$), the game reduces to two decoupled single-agent optimal control; the empirical mean trajectories follow their respective moving wells while maintaining proximity to the baseline reference. In the repulsive regime ($\gamma = 0.6$), the cross-log-likelihood cost penalizes overlapping distributions. Consequently, players exhibit proactive congestion avoidance, taking wider, sub-optimal tracking routes to maintain a spatial buffer. Conversely, in the attractive regime ($\gamma = -0.6$), players actively compromise their individual state cost goals to stay closer to the origin, increasing the shared probability mass. Finally, Fig. \ref{fig:distance} quantifies the temporal separation between these distributions, while Fig. \ref{fig:terminal_result} illustrates the spatial deviation from the terminal target induced by the cross-log-likelihood coupling. We also consider an additional asymmetric coupling regime in which the off-diagonal interaction terms have opposite signs defined by $\alpha(\gamma) = \left[\begin{smallmatrix} 1 & -\gamma \\ \gamma & 1 \end{smallmatrix}\right]$. Unlike the symmetric benchmark cases, this non-reciprocal regime penalizes one player for overlap while the other is encouraged toward it, as illustrated in Fig. \ref{fig:asymmetric_measures}. This shows that the proposed game framework captures reciprocal behaviors, such as mutual congestion avoidance or cohesion, while also accounting for nonreciprocal interactions like pursuit-evasion.
\vspace{-10pt}
\section{Conclusion}
We presented a class of measure-theoretic general-sum dynamic game and it's equivalent continuous time nonlinear stochastic differential game. We showed that the resulting coupled nonlinear
HJB equations can be exactly decoupled and linearized via a
multivariate Cole-Hopf transformation.  The linearized system admits a Feynman-Kac path-integral representation, allowing Nash equilibrium strategies to be computed through forward Monte Carlo sampling without spatial discretization, thereby circumventing the curse of dimensionality. Simulations on a two-player problem show that the proposed game can capture reciprocal behaviors like congestion avoidance or cohesion, as well as asymmetric interactions like pursuit-evasion at the distributional level.
\vspace{-15pt}
\bibliographystyle{IEEEtran}
\bibliography{bibliography}
\end{document}